\documentclass[11pt]{amsart}
\usepackage{amsfonts,amsmath,amsthm,amssymb, color,amscd}
\usepackage{fullpage}
\usepackage{todonotes}

\numberwithin{equation}{section}

\newtheorem{thm}{Theorem}[section]
\newtheorem*{thm*}{Theorem}
\newtheorem{prop}[thm]{Proposition}
\newtheorem*{prop*}{Proposition}

\newtheorem{cor}[thm]{Corollary}

\newtheorem{defin}[thm]{Definition}

\newtheorem{lemma}[thm]{Lemma}

\newtheorem{remark}[thm]{Remark}
\newtheorem*{remark*}{Remark}
\newtheorem*{remarks*}{Remarks}

\newcommand{\ip}[1]{\langle #1 \rangle}

\newcommand{\restrictto}[2]{\left. #1 \right|_{#2}}
\newcommand{\ddtat}{\restrictto{\frac{d}{dt}}{t=0}}

\newcommand{\tr}{\operatorname{tr}}
\newcommand{\Ric}{\operatorname{Ric}}
\newcommand{\ric}{\operatorname{ric}}

\newcommand{\Ker}{\operatorname{Ker}}
\newcommand{\Aut}{\operatorname{Aut}}
\newcommand{\Der}{\operatorname{Der}}
\newcommand{\Hom}{\operatorname{Hom}}

\begin{document}

\title{A Step Towards the Alekseevskii Conjecture}
\author{Michael Jablonski and Peter Petersen}

\maketitle

\begin{abstract}  We refine existing structure results  for non-compact, homogeneous, Einstein manifolds and provide a reduction in the classification problem of  such spaces.  Using this work, we verify the (Generalized) Alekseevskii Conjecture for a large class of homogeneous spaces.
\end{abstract}

A longstanding open question in the study of Riemannian homogeneous spaces is the classification of non-compact, Einstein spaces.  In the 1970s, it was conjectured by D.~Alekseevskii that any (non-compact) homogeneous Einstein space of negative scalar curvature is diffeomorphic to $\mathbb R^n$.  Equivalently, this conjecture can be phrased as follows:

\begin{quote}
\textbf{Classical Alekseevskii Conjecture:}  Given a  homogeneous Einstein space $G/K$ with negative scalar curvature, $K$ must be a maximal compact subgroup of $G$. 
\end{quote}
Part of the motivation for this conjecture comes from the conclusion holding true in the special cases of simply-connected Ricci flat homogeneous spaces \cite{AlekseevskiiKimelfeld:StructureOfHomogRiemSpacesWithZeroRicciCurv}, non-compact symmetric spaces, and, more generally, homogeneous (Einstein) spaces of negative sectional curvature \cite{Alekseevski:RiemSpacesOfNegCurv,AzencottWilson:HomogMfldsWithNegCurv}.   It is notable that in all these cases  the manifolds are so-called solvmanifolds, i.e.~they admit a transitive solvable group of isometries.

Since the posing of this conjecture, an enormous effort has been put into the classification of non-compact, homogeneous Einstein spaces.  Among simply-connected solvable Lie groups with left-invariant metrics, much is known about existence and uniqueness of Einstein metrics, see   \cite{Heber,LauretStandard} and references therein.  Further, it follows from \cite{AlekseevskyCortes:IsomGrpsOfHomogQuatKahlerMnflds} that in the case of Einstein solvmanifolds with negative scalar curvature, one can reduce to the simply-connected case; see \cite{Jablo:StronglySolvable} for more in this direction.  

In contrast to the progress made on solvmanifolds, there is work that suggests the conjecture might not be true.  In \cite{Dotti-Leite:MetricsOfNegativeRicciCurbatureOnSLnR} it is shown that $SL(n,\mathbb R)$ admits metrics of negative Ricci curvature for $n\geq 3$; although, in those examples the eigenvalues of $\Ric$ appear to be so widely spread that Einstein metrics might not exist.  The case of transitive, unimodular groups of isometries was further investigated in \cite{DottiMiatello:TransitveGroupActionsAndRicciCurvatureProperties} where it was shown that for such a space with negative Ricci curvature there must exist a transitive semi-simple group of isometries.  For more in this direction, see \cite{Nikonorov:OnTheRicciCurvatureOfHomogeneousMetricsOnNoncompactHomogeneousSpaces} and   Proposition \ref{prop: G semi-simple no compact factors} below.

The conjecture is known to hold in dimensions 4 and 5 and, in fact, all such spaces in these low dimensions are solvmanifolds, see \cite{Jensen:HomogEinsteinSpacesofDim4,Nikonorov:NoncompactHomogEinstein5manifolds}.  In dimension 6, very recently it was shown that the conjecture holds in the presence of a transitive, non-unimodular group of isometries, see \cite{Arroyo-Lafuente:HomogeneousRicciSolitonsInLowDimensions} or Section \ref{sec: GAC in low dim} of this work.  Again, all such 6-dimensional spaces turn out to be solvmanifolds.

In the general setting, very little was known about the structure of homogeneous Einstein metrics until the recent work \cite{LauretLafuente:StructureOfHomogeneousRicciSolitonsAndTheAlekseevskiiConjecture}.  The work presented here builds on the structure results obtained there.

\subsection*{Refining the conjecture.}  We refine the question of which $G$ and $K$ are possible when $G/K$ admits a $G$-invariant Einstein metric of negative scalar curvature.    More precisely, assume $G$ is simply-connected and consider a Levi decomposition 
	$$G=G_1\ltimes G_2$$
where $G_1$ is semi-simple and $G_2$ is the (solvable) radical.  Further, decompose $G_1$ into compact and non-compact factors, i.e.
	$$G_1 = G_c G_{nc}$$
where $G_c$ is the product of compact, simple subgroups and $G_{nc}$ the product of the non-compact, simple subgroups.  The intersection $G_c \cap G_{nc}$ is trivial since we assumed $G$ to be simply-connected.  Let $K_{nc}$ be a choice of connected subgroup whose Lie algebra is a maximal compact subalgebra of $\mathfrak g_{nc}$ (equivalently, $Ad(K_{nc})$ is a maximal compact subgroup of $Ad(G_{nc})$).  The subgroup $K_{nc}$ is closed, although not necessarily compact. 

	\begin{quote} \textbf{Strong Alekseevskii Conjecture:} Let $G/K$ be endowed with a $G$-invariant Einstein metric with negative scalar curvature.  Then, up to conjugation of $K$ in $G$,  $G_c < K$ and $K_{nc} < K$.
	\end{quote}

Observe that this is a stronger conjecture than  the Alekseevskii Conjecture as the above implies that such an  Einstein manifold has a transitive solvable group of isometries.  We note that all known examples of homogeneous Einstein spaces with negative scalar curvature are isometric to solvable Lie groups with left-invariant metrics and it is believed by some that such spaces exhaust the class of non-compact, homogeneous, Einstein spaces.  Further, the above conjecture holds for any transitive group of isometries of an Einstein solvmanifold, see \cite{Jablo:StronglySolvable}.   We take a step towards resolving the above conjecture.

\begin{thm}\label{thm: main result Gc < K}  Let $G/K$ be a  homogeneous space endowed with a $G$-invariant Einstein metric of negative scalar curvature, then  $G_c <K$ (up to conjugation).
\end{thm}

This result is established by applying the recent structural results of \cite{LauretLafuente:StructureOfHomogeneousRicciSolitonsAndTheAlekseevskiiConjecture} together with the Bochner technique and tools from Geometric Invariant Theory.  Building on that work and the above theorem, we have the following reduction in the classification problem.

\begin{thm}\label{thm: reduction to nicer G1=Gnc} Let $G/K$ be a simply-connected, homogeneous, Einstein space of negative scalar curvature.  The transitive group $G$ can be chosen to satisfy the   following 
	\begin{enumerate}
	\item $G_1 = G_{nc}$ has no compact, normal subgroups, 
	\item $K< G_1$, and 
	\item The radical decomposes as $G_2 = AN$, where the nilradical $N$ (with the induced left-invariant metric) is nilsoliton, $A$ is an abelian group,  and $ad~\mathfrak a$ acts by symmetric endomorphisms relative to the nilsoliton metric on $\mathfrak n$.
	\end{enumerate}
\end{thm}

We note that in the special case that $G$ is solvable, (iii) above was already proven by Lauret in \cite{Lauret:SolSolitons}. 

\begin{remark}  Applying \cite{Jablo:HomogeneousRicciSolitonsAreAlgebraic} together with either of \cite{HePetersenWylie:WarpProdEinsteinMetricsOnHomogAndHomogRicciSolitons} or \cite{LauretLafuente:StructureOfHomogeneousRicciSolitonsAndTheAlekseevskiiConjecture}, the above two theorems can be seen to apply more generally to homogeneous Ricci solitons.
\end{remark}

As an application of Theorem \ref{thm: reduction to nicer G1=Gnc}, we obtain a short proof of the recent result of Arroyo-Lafuente which verifies the Generalized Alekseevskii Conjecture in low dimensions, see Section \ref{sec: GAC in low dim}.

\section{The maximum principle}

To motivate our new results, we begin by recalling the following well-known fact due to Bochner \cite{Bochner:VectorFieldsAndRicciCurvature}.

\begin{thm}[Bochner]  Let $M$ be a compact Riemannian manifold with negative Ricci curvature, then the isometry group of $M$ is discrete.
\end{thm}

One consequence of this result is that a homogeneous  space with negative Ricci curvature is necessarily non-compact.  If such a space admits a transitive unimodular group of isometries, then it is known to admit a transitive semi-simple group of isometries \cite{DottiMiatello:TransitveGroupActionsAndRicciCurvatureProperties}.  However, little else is known about these spaces in general, save some very interesting cases worked out by Nikonorov \cite{Nikonorov:OnTheRicciCurvatureOfHomogeneousMetricsOnNoncompactHomogeneousSpaces} 

We  follow Bochner's approach to glean even more about the geometry of homogeneous spaces with negative Ricci curvature.

\begin{prop}\label{prop: G semi-simple no compact factors} Let $G$ be a semi-simple group and assume that $G/K$ is a homogeneous space of negative Ricci curvature.  If $G$ acts almost effectively on $G/K$, then $G$ is of non-compact type, i.e.\  $G$ has no (non-discrete) compact, normal subgroups.
\end{prop}

\begin{cor}\label{cor: isometry group of G/K with negative Ric} Let $G$ be a semi-simple group and $G/K$  a homogeneous space of negative Ricci curvature on which $G$ acts effectively.  Then  the connected isometry group of $G/K$ is simply $G$ and $G$ is of non-compact type.
\end{cor}

Corollary \ref{cor: isometry group of G/K with negative Ric} follows immediately from Proposition \ref{prop: G semi-simple no compact factors} together with a general result of Gordon on the isometry group of homogeneous spaces where there exists a transitive semi-simple group of non-compact type, see \cite{Gordon:RiemannianIsometryGroupsContainingTransitiveReductiveSubgroups}.  We note that Corollary \ref{cor: isometry group of G/K with negative Ric} generalizes \cite[Corollary 3]{DottiMiatello:TransitveGroupActionsAndRicciCurvatureProperties}.  This gives some hope that the Alekseevskii Conjecture is indeed true as one would expect an Einstein space to have more symmetries than other metrics and there are metrics on $G/K$ whose symmetry groups are as large as  $G\times K$.  This philosophy of Einstein spaces having a large amount of symmetry is reinforced by a recent result of the first author and Gordon \cite{GordonJablonski:EinsteinSolvmanifoldsHaveMaximalSymmetry}.

In the special case of semi-simple Lie groups with left-invariant metrics, the above gives new information.  In the compact case, it is well-known that all semi-simple groups admit Einstein metrics, often more than one \cite{Jensen:TheScalarCurvatureOfLeftInvariantRiemannianMetrics}.  In the non-compact case, one can quickly see by brute force that $SL_2\mathbb R$ does not admit a left-invariant Einstein metric.  Until now, this was the only non-compact, semi-simple  group for which the existence question had been answered.

\begin{cor}  Let $G$ be a non-compact, semi-simple Lie group.  If $G$ has a (non-discrete) compact, normal subgroup, then $G$ does not admit a left-invariant Einstein metric.
\end{cor}

To prove the  proposition above, we appeal to a slightly more general setting that will be needed later.  We begin with a general lemma which is well-known and  apply it to the homogeneous setting.

\begin{lemma}\label{lemma: maximum principle} Let $X$ be a Killing field on a Riemannian manifold $M$ and consider the function $f =\frac{1}{2}  |X|^2$.  Then 
	$$\Delta f = | \nabla X |^2 - \ric(X).$$
Furthermore, let $f(p)$ be a maximum of $f$ and assume that $X_p$ is tangent to a subspace of $T_pM$ along which the Ricci tensor is negative. By the maximum principle, we  have  $X=0$.
\end{lemma}

We apply this lemma in the special case that $M=G/K$.   Let $G=G_1G_2$ be a Levi-decomposition of $G$ and consider the representation $\theta: \mathfrak g \to \Der(\mathfrak g_2)$  defined by 
	$$\theta (X) = ad(X)|_{\mathfrak g_2}.$$
Let $\Ker~\theta = \{ X\in \mathfrak g \ | \ ad(X)|_{\mathfrak g_2} = 0 \}$.  As $\mathfrak g = \mathfrak g_1\ltimes \mathfrak g_2$ is a Levi decomposition of $\mathfrak g$,  we see that $\theta(\mathfrak g_1)\ltimes \theta(\mathfrak g_2)$ is a Levi decomposition of $\theta(\mathfrak g)$ and so 
$\theta(\mathfrak g_1) \cap \theta( \mathfrak g_2)$ is trivial.  This yields 
	$$\Ker~\theta = (\Ker~\theta \cap \mathfrak g_1) + ( \Ker~\theta \cap \mathfrak g_2).$$
Note that the subalgebra   $\Ker~\theta \cap \mathfrak g_1$ is an ideal of $\mathfrak g_1$.

As above, we may decompose $\mathfrak g_1 = \mathfrak g_c + \mathfrak g_{nc}$ into a sum of ideals where $\mathfrak g_c$ is the sum of the compact, simple ideals and $\mathfrak g_{nc}$ is the sum of the non-compact, simple ideals.  The ideal $\Ker~\theta \cap \mathfrak g_1$ decomposes further into $\Ker~\theta \cap \mathfrak g_1 = (\Ker~\theta \cap \mathfrak g_c) + (\Ker~\theta \cap \mathfrak g_{nc})$.  Our interest is in the connected (normal) subgroup $C$ of $G$ with Lie algebra
	$$Lie~C = (\Ker~\theta \cap \mathfrak g_c) + \mathfrak{z(g)},$$
where $\mathfrak{z(g)}$ denotes the center of $\mathfrak g$.  Note that $Lie~C \subset \Ker~\theta$.

Observe that  $G$ may be described as a product
	$$G= C D,$$
where $D$ is a subgroup of $G$ which commutes with $C$ and such that $C\cap D = Z(G)$.  To see this, one builds $D$ from $G_2$ and the normal subgroups of $G_1$ which do not appear in $C$.  Note, in the case  $G$ is semi-simple, i.e. $G_2$ is trivial, we have  $\mathfrak g = \mathfrak g_1  =\Ker~\theta$.

\begin{lemma}\label{lemma: negative Ricci curv in directions tangent to compact normal}  
There exists no $G$-invariant metric on $G/K$ whose Ricci curvature is negative in the directions  tangent to the orbit $CK = C\cdot  eK \subset G/K$. 
\end{lemma}

\begin{remark}
If $G/K$ were endowed with a metric whose  Ricci curvature was negative along the orbit $C\cdot  eK \subset G/K$ (as long as the orbit were non-trivial), then we'd necessarily have that  $C< K$ and so the orbit would indeed be trivial for these examples.  Notice, this proves  Theorem \ref{thm: main result Gc < K} in the special case that $G$ is semi-simple.  Further, if $G$ were acting almost effectively, then $C$  must be the trivial group.
\end{remark}

Before proving this lemma, we state a corollary which has not appeared in the literature, but is  known to some experts.

\begin{cor}\label{cor: negative ricci implies no center} Let $G/K$ be a homogeneous space where $G$ acts almost effectively.  In directions tangent to the orbit of $Z(G)$, the center of $G$, the Ricci curvature is non-negative.   Furthermore, if $G/K$ is endowed with a metric of negative Ricci curvature, then $Z(G)$ is discrete.  
\end{cor}

In the special case of left-invariant metrics on Lie groups, this result is well-known \cite{Milnor:LeftInvMetricsonLieGroups}.  In the general homogeneous case, one can   deduce an alternate proof to the corollary above using the techniques in \cite{DottiMiatello:TransitveGroupActionsAndRicciCurvatureProperties}.  Additionally, Jorge Lauret has shown us a different proof which uses the relationship between $\Ric$ and the moment map; Yurii Nikonorov has pointed out to us that the corollary also follows from \cite[Theorem 4]{BerestovskiNikonorov:KillingVectorFieldsOfConstantLengthOnRiemannianManifolds}.

\begin{proof}[Proof of Lemma \ref{lemma: negative Ricci curv in directions tangent to compact normal}]  
Assume that $C\cdot  eK \subset G/K$ is non-trivial and  $G/K$ is endowed with a metric such that the Ricci curvature is negative in directions tangent to the orbit $C\cdot  eK \subset G/K$.

The proof of this lemma follows quickly from Lemma \ref{lemma: maximum principle}.  To apply that result, we take $X\in Lie~C$ and consider the Killing field $[X]$ generated by $X$; i.e. 
	$$ [X]_p = \ddtat exp (tX)\cdot p.$$
As $p = gK$ for some $g\in G$, and $g=c d$ for $c\in C$ and $d\in D$, we see that
	$$ | [X]_{gK} | = | Ad(g^{-1}) X | = | Ad(c^{-1}) X | .$$
Using that the group $C$ is $Ad$-compact, we see  that  the function $f = \frac{1}{2} | [X] |^2$ does achieve a maximum.  Further,  by replacing $X$ with $Ad(c^{-1})X \in Lie~C$, we may assume that this maximum occurs at the point $eK$.

By hypothesis, the Ricci curvature is negative along the orbit $C\cdot eK$ and   $\ric( Ad(c^{-1})X ) < 0$, unless $[X]=0$.  Applying  Lemma \ref{lemma: maximum principle}, we see that $[X]=0$.  Thus $C<K$, which is a contradiction.

In the case that $G$ acts effectively,  $[X]=0$ implies $X=0$, and thus $C$ must be  trivial as it is connected.
\end{proof}

We apply the results above in the special case of Einstein metrics to obtain our main results.  First, we introduce our next  tool, the moment map.

\section{The moment map}\label{section: the moment map}

Let $G$ be a real reductive Lie group acting linearly on a real vector space $V$.  Denoting the $G$-action on $V$ by $\rho$, we assume that $V$ is endowed with an inner product $\ip{\cdot, \cdot }$ with the property
	$$\rho(g)^t \in \rho(G) \quad \mbox{ for all } g\in G,$$
where $\cdot ^t$ denotes the transpose relative to $\ip{\cdot,\cdot }$.  Such inner products always exist for semi-simple $G$ and, more generally, whenever $\rho(G) \subset Aut(V)$ is algebraically closed and fully-reducible, see \cite{Mostow:SelfAdjointGroups}.  In this setting, we say   $G$ is \textit{self-adjoint} with respect to $\ip{\cdot,\cdot}$.

We may endow $\mathfrak g = Lie~G$ with an inner product $\ip{\ip{\cdot,\cdot}}$ such that $Ad (G)$ is self-adjoint relative to $\ip{\ip{\cdot,\cdot}}$.  Using these choices of inner products, and inspired by moment maps from symplectic and complex geometry \cite{Ness}, we define the moment map $m:V\to \mathfrak g$ of the $G$ action on $V$ by
	$$\ip{\ip{ m(v) ,X}} = \ip{\rho(X)v,v},$$
for all $X\in\mathfrak g$ and $v\in V$.  We note that we have abused notation and written the induced Lie algebra representation of $\mathfrak g = Lie~G$ by the same symbol.  Our definition of the moment map given above for a real reductive Lie group is fairly standard, although some authors would differ  in that $m$ would take values in the dual space $\mathfrak g^*$. 

\begin{remark} In the sequel, we suppress $\rho$ and denote $\rho(g)v$ by $g\cdot v$.
\end{remark}

\begin{defin} A point where the moment map vanishes is called a minimal point.\end{defin}

\begin{thm}\cite[Thm 4.3]{RichSlow}\label{thm: Rich-Slow} Let $G$ be a reductive group acting linearly on $V$, as above, and $p$ be a minimal point.  The orbit $G\cdot p$ is closed and  the stabilizer subgroup $G_p$ is self-adjoint, i.e. closed under transpose.  Furthermore, if $q\in G\cdot p$ is another minimal point, then $q\in K\cdot p$, where $K = G \cap O(V)$ and $O(V)$ is the orthogonal group relative to the inner product on $V$.
\end{thm}

We note that in the work of Richardson and Slodowy \cite{RichSlow}, they do not define the moment map.  However, their definition of minimal point coincides with ours.  Our interest in the tools above comes from looking at the change of basis action on the space of representations of a Lie algebra.

\subsection{The space of representations}  Given a Lie algebra  $\mathfrak n$  and vector space $\mathfrak h$, we consider the vector space  $V=\Hom(\mathfrak h, \Der(\mathfrak n))$.  If $\mathfrak h$ were a Lie algebra acting by derivations on $\mathfrak n$, we could think of this representation of $\mathfrak h$ as an element of the vector space  $V=\Hom(\mathfrak h, \Der(\mathfrak n))$.  In the sequel, $\mathfrak h$ will not be a Lie algebra itself, but a special subspace of a Lie algebra.

The group  $GL(\mathfrak n)$ acts in the natural way on $V$.  For $\theta \in V$ and $g\in GL(\mathfrak n)$ we define $g\cdot \theta$ by
	$$(g\cdot \theta)(Y) = g \theta(Y) g^{-1},$$
for $Y\in\mathfrak h$.  
To define a moment map, we consider the most natural inner products on $V$ and $\mathfrak{gl}(\mathfrak n)$.

On $\mathfrak{gl(n)}$, we use the usual inner product $\ip{\ip{A,B}} = \tr~AB^t$.  To define an inner product on $V$, we first assume that $\mathfrak n$ and $\mathfrak h$ are endowed with inner products.  Now take $\theta, \lambda \in V$ and $\{Y_i\}$ an orthonormal basis of $\mathfrak h$ and define
	$$\ip{\theta,\lambda} = \sum_i \tr~\theta(Y_i)\lambda(Y_i)^t,$$
where the transpose is being taken with respect to the inner product on $\mathfrak n$.  To be able to define a moment map, we must first observe that indeed $\mathfrak{gl(n)}$ is self-adjoint.

\begin{remark}  If we denote this representation of $GL(\mathfrak n)$ on $V$ by $\rho$, then we have  
	$$\rho(g^t) = \rho(g)^t,$$
where the first transpose is taken in $GL(\mathfrak n)$ relative to the inner product on $\mathfrak n$ and the second is taken in $GL(V)$ relative to our natural choice of inner product on $V=  \Hom(\mathfrak h,\Der(\mathfrak n))$.  In this way, the action of $GL(\mathfrak n)$ on $V$ is self-adjoint.
\end{remark}

One quickly sees that the moment map $m:V\to \mathfrak{gl(n)}$ of the $GL(\mathfrak n)$ action on $V$  is given by
	$$m(\theta) = \sum_i [\theta(Y_i),\theta(Y_i)^t].$$
This is also computed in the appendix of \cite{Arroyo-Lafuente:HomogeneousRicciSolitonsInLowDimensions} written by Jorge Lauret.

\section{Einstein metrics on $G/K$}\label{Section: Einstein metrics}
Our main structure results on non-compact, homogeneous Einstein spaces build on those recently obtained in \cite{LauretLafuente:StructureOfHomogeneousRicciSolitonsAndTheAlekseevskiiConjecture}.  We begin by recalling the details from that work which we will use.  In this section $G$ will be a Lie group, not necessarily reductive.

Let $G/K$ be endowed with an Einstein metric of negative scalar curvature.  The Lie algebra $\mathfrak g = Lie~G$ admits a decomposition which is akin to an `algebraic Levi decomposition'.  Namely, $\mathfrak g = \mathfrak u \ltimes \mathfrak n$, where $\mathfrak u$ is reductive and $\mathfrak n$ is the nilradical.  In terms of the above notation, we have
	$$\mathfrak g_1 = [\mathfrak u,\mathfrak u]  \quad \mbox{ and } \quad  \mathfrak g_2 = \mathfrak{z(u)} \ltimes \mathfrak n,$$
where $\mathfrak{z(u)}$ denotes the center of $\mathfrak u$ and $\mathfrak n$ is the nilradical of $\mathfrak g$.  Furthermore, $\mathfrak k \subset \mathfrak u$ and there exists an $Ad~K$-stable complement $\mathfrak h$ of $\mathfrak{k}$ in $\mathfrak{u}$.  Naturally, we identify $\mathfrak h \oplus \mathfrak n$ with $T_{eK}G/K$ and we have that $\mathfrak h \perp \mathfrak n$.  These properties are very special to the Einstein setting and do not happen in general, see \cite{LauretLafuente:StructureOfHomogeneousRicciSolitonsAndTheAlekseevskiiConjecture} for more details.

\begin{remark} The inner product on $\mathfrak h$ extends naturally to an inner product on $\mathfrak u$ such that $\mathfrak h \perp \mathfrak k$ with $Ad(K)$ acting orthogonally.  As $\mathfrak u$ is reductive, we have $\mathfrak u = [\mathfrak u , \mathfrak u] + \mathfrak{z(u)}$ (vector space direct sum).  Recently, Arroyo and Lafuente proved that this direct sum is orthogonal using Lemma \ref{lemma: reducing the compatibility condition} (i) below, see \cite{Arroyo-Lafuente:TheAlekseevskiiConjectureInLowDimensions}.
\end{remark}

The following comes from \cite[Theorem 4.6]{LauretLafuente:StructureOfHomogeneousRicciSolitonsAndTheAlekseevskiiConjecture}.  In the sequel, $G$ is a simply-connected Lie group, $K$ is connected, and so $G/K$ is simply-connected.

\begin{thm}[Lafuente-Lauret]\label{thm: LL Einstein metrics structure} Let $G/K$ be a  homogeneous space endowed with an Einstein metric of negative scalar curvature.  Then $G=U\ltimes N$, where $U$ is reductive, $N$ is nilpotent, and $K<U$.  Furthermore,
	\begin{enumerate}
	\item The induced left-invariant metric on $N$ is nilsoliton.
	
	\item Denote the adjoint action of $\mathfrak u$ on $\mathfrak n$ by $\theta : \mathfrak u \to  \Der (\mathfrak n)$.  	The induced metric on $U/K$ satisfies $\Ric_{U/K} = cId +C_\theta$, where $C_\theta$ is  the symmetric operator   defined by
	$$\ip{C_\theta(Y),Y}  = \frac{1}{4} tr (S(\theta(Y) )^2  .$$
Here $S(A)$ denotes the symmetric part of the endomorphism $A: \mathfrak n \to \mathfrak n$ relative to the soliton metric on $N$.

	\item The adjoint action of $\mathfrak u$ on $\mathfrak n$  satisfies the compatibility condition 
		$$ \sum_i [\theta(Y_i) , \theta(Y_i)^t] =0$$
		where $\{Y_i\}$ is an orthonormal basis of $\mathfrak h \subset \mathfrak u$ and transpose is being taken with respect to the nilsoliton metric on $\mathfrak n$.
	\end{enumerate}

\end{thm}

Using this theorem and the following technical lemma, to be proven later, we prove Theorem \ref{thm: main result Gc < K}.

\begin{lemma}\label{lemma: theta gc is skew-symmetric} Let $G=U\ltimes N$ act transitively on an Einstein space of negative scalar curvature.  Let $\mathfrak g_1 = [\mathfrak u,\mathfrak u]$ and write $\mathfrak g_1 = \mathfrak g_c + \mathfrak g_{nc}$ as  a sum of its compact and non-compact ideals.  Then $\theta(\mathfrak g_c)$ consists of skew-symmetric endomorphims acting on $\mathfrak n$.
\end{lemma}

\begin{proof}[Proof of Theorem \ref{thm: main result Gc < K}]  We first note that it suffices to prove the claim in the case that $G/K$ is simply-connected.  This follows from the fact that the simply-connected cover  $\tilde G$ of $G$ acts on the simply-connected cover of $G/K$, that $\tilde G_c$ covers $G_c$, and $\tilde K$ covers $K$.  So if $\tilde G_c < \tilde K$, then $G_c < K$.

From the above theorem, we have
	$$\Ric_{U/K} = cId +C_\theta$$
where$\ip{C_\theta(Y),Y}  = \frac{1}{4} tr (S(\theta(Y) )^2 $.  The lemma above shows that $C_\theta$ vanishes on $\mathfrak g_c$ and so in the direction of any non-vanishing Killing field $Y\in \mathfrak g_c$, we see that
	$$\Ric_{U/K} (Y) < 0.$$
Applying Lemma \ref{lemma: negative Ricci curv in directions tangent to compact normal} to $U/K$, we see that this is not possible and so $G_c < K$, as desired.  This proves Theorem \ref{thm: main result Gc < K}.

\end{proof}

\subsection{Refining the structure of $G$ and $\theta$}
To prove Lemma \ref{lemma: theta gc is skew-symmetric} and construct a transitive group with the properties given in Theorem \ref{thm: reduction to nicer G1=Gnc}, we first refine the structure of the $U$ action on $\mathfrak n$.

\begin{remark} Throughout, we   consider the adjoint representation of $U$ on $\mathfrak n$.  We abuse notation and denote this by $\theta : U \to \Aut(\mathfrak n)$.   This choice of notation is natural as  the corresponding Lie algebra representation of $\mathfrak u$ is precisely the one denoted by $\theta : \mathfrak u \to \Der(\mathfrak n)$ above.
\end{remark}

We may restrict $\theta$ to be a map $\theta: \mathfrak h \to \Der(\mathfrak n)$.  Now the compatibility condition above says precisely that $\theta$ is a minimal point of the $GL(\mathfrak n)$ action on $V = \Hom(\mathfrak h, \Der(\mathfrak n))$.

\begin{remark}
Observe that by extending the inner product on $\mathfrak h \oplus \mathfrak n$ to an inner product on $\mathfrak g = \mathfrak u \oplus \mathfrak n$, as described above, we have that $ad~\mathfrak k$ acts skew-symmetrically and so  $\theta: \mathfrak u \to \Der(\mathfrak n)$ is a minimal point of the $GL(\mathfrak n)$ action on $V = \Hom(\mathfrak u,\Der(\mathfrak n))$.  We adopt this view in the sequel.
\end{remark}

\begin{lemma}\label{lemma: reducing the compatibility condition}  The abelian algebra $\theta(\mathfrak{z(u)})$ consists of normal operators such that $\theta(\mathfrak{z(u)})^t$ commutes with all of $\theta(\mathfrak u)$.
\end{lemma}

\begin{proof}  This is simply an application of more general GIT results \cite[Thm 4.3]{RichSlow} (Theorem \ref{thm: Rich-Slow} above).  To apply that work, we compute the stabilizer of the $GL(\mathfrak n)$-action at $\theta$ which is 
	$$GL(\mathfrak n) _\theta = \{g\in GL(\mathfrak n) \ | \ g\cdot \theta = \theta \}  = \{ g\in GL(\mathfrak n) \ | \ 
						g\theta(Y) g^{-1} = \theta(Y) \mbox{ for all } Y\in \mathfrak u \} , $$
that is, $g\in GL(\mathfrak n)_\theta$ must commute  with all $\theta(Y)$.  Now take $u\in Z(U)$ and consider $g={\theta(u)}$.  Clearly ${\theta(u)}\in (GL(\mathfrak n))_\theta$ and so, by the theorem above, we have the same for $({\theta(u)})^t$.  Upon differentiating, we obtain the  lemma.
\end{proof}

We now extend the group $G$ to a larger, transitive group $\overline G$ of isometries whose structure is somewhat cleaner.  For  $u\in Z(U)$, we consider $\theta(u)^t$   and $\phi(u) = \theta(u) (\theta(u)^t)^{-1} \in \Aut(N)$.  
From the above, we know that $\phi(u)$ acts orthogonally and  commutes with $\theta(U)$.  Thus we may realize $\phi(u) \in \Aut(G)$ and consider the group
	$$\overline G = \phi(Z(U)) \ltimes G$$
which acts naturally on $G/K$ by isometries and with stabilizer $\overline K = \phi(Z(U))  K$.  

\begin{remark}\label{remark: g-bar the extended group}  By construction, $\overline G = \overline U \ltimes N$, where $\overline U = U \phi(Z(U))$, $G_1 = [U,U]=[\overline U,\overline U]$, and $Z(\overline U) = Z(U)\phi(Z(U))$.  From Lemma \ref{lemma: reducing the compatibility condition}, we have that $\theta(Z(\overline U))$ is closed under transpose and we will see below that $\overline U$ is the smallest group containing $U$ such that $\theta(\overline U)$ is self-adjoint.
\end{remark}

\begin{remark} Denote the Lie algebra of $\overline U$ by $\overline{\mathfrak{u}}$.  As $\phi(\mathfrak{z(u)})$ acts skew-symmetrically on $\mathfrak n$ and trivially on $\overline{\mathfrak{u}}$,  we may extend the inner product on $\mathfrak g$ to one on $\overline{\mathfrak g}$  so that  $\phi(\mathfrak{z(u)})$  acts skew-symmetrically on $\overline{\mathfrak{g}}$.  In doing so,   the extension of $\theta$ to all of $\overline{\mathfrak{u}}$ is a minimal point in $\Hom (\overline{\mathfrak u}, \Der( \mathfrak n))$.
\end{remark}

\begin{prop}\label{prop:  U bar acts self-adjointly on nilrad} Let $\theta$ be a minimal point of the $GL(\mathfrak n)$-action on $V=Hom(\overline{\mathfrak u}, Der(\mathfrak n))$, as above.   Then   
	\begin{enumerate}
	\item $\theta(\overline{\mathfrak u})$ is self-adjoint,
	\item $\theta(\mathfrak g_1)$ is self-adjoint, and
	\item $\theta(\mathfrak g_c)$ acts skew-symmetrically  on $\mathfrak n$.
	\end{enumerate}

\end{prop}

Before proving (i), we use it to quickly justify the last two claims.  To see that $\theta(\mathfrak g_1)$ is self-adjoint, observe that this subalgebra  is precisely the commutator subalgebra of $\theta(\overline{\mathfrak{u}})$ and the commutator subalgebra of a self-adjoint algebra is always self-adjoint.  This proves (ii).

To see (iii), note that  $\theta(\mathfrak g_1) \cap \mathfrak{so}(\mathfrak n)$ is a maximal compact subalgebra of $\theta(\mathfrak g_1)$ as $\theta(\mathfrak g_1)$ is self-adjoint.  Recall that the maximal compact subalgebras of $\theta (\mathfrak g_1)$ are all conjugate.  As $\mathfrak g_1 = \mathfrak g_c + \mathfrak g_{nc}$, we see that $\theta( \mathfrak g_c  )$ is contained in every maximal compact subalgebra and so we have that $\theta(\mathfrak g_c)$ is contained in $\mathfrak{so(n)}$.  Note, (iii) above is the statement in Lemma \ref{lemma: theta gc is skew-symmetric}.

\begin{remark}  It seems noteworthy to point out  one consequence of the proposition above.  If  $G/K$ admits a $G$-invariant, Einstein metric of negative scalar curvature, then from Theorem \ref{thm: LL Einstein metrics structure}  the induced geometry on $U/K$ will be such that the most negative eigenvalue of Ric occurs along the orbit of a maximal compact subgroup of $G_1$.  No known examples of this kind exist. There are examples of non-compact semi-simple Lie groups $G_1$ with left-invariant metrics with negative Ricci curvature, but in these examples  the eigenvalues in the direction of the maximal compact of $G_1$ are the closest to being zero, not the most negative.
\end{remark}

To prove part (i) of the proposition, we first establish some preliminary results.

\begin{lemma}\label{lemma: mostow on selfadjointness} There exists $g\in GL(\mathfrak n)$ such that $(g\cdot \theta)(  \overline{\mathfrak u})$ is self-adjoint.
\end{lemma}

The above is equivalent to saying that there exists some inner product relative to which  $\theta( \overline{\mathfrak u})$ is self-adjoint.  The existence of such an inner product is a classical result of Mostow \cite{Mostow:SelfAdjointGroups}.  To apply that work, observe that  $\theta( \overline{\mathfrak u}) = \theta([\mathfrak u,\mathfrak u]) + \theta (\mathfrak{z(\overline{u})})$,  the Lie algebra $\theta([\mathfrak u,\mathfrak u])$ is semi-simple and hence fully reducible and algebraic, and $\theta(\mathfrak{z(\overline{u})})$ was shown to be self-adjoint above, Lemma \ref{lemma: reducing the compatibility condition}.


\begin{lemma}  Let $\theta_0 \in Hom(\overline{\mathfrak{u}}, \Der(\mathfrak n))$ be such that $\theta_0(\overline{\mathfrak{u}})$ is self-adjoint.  Let $\theta_t$ be the trajectory of the negative gradient flow of $||m||^2$, where $m$ is the moment map for the $GL(\mathfrak n)$ action on $Hom(\overline{\mathfrak{u}}, \Der(\mathfrak n))$.  Then $\theta_t(\overline{\mathfrak u}) = \theta_0(\overline{\mathfrak u}) $ is self-adjoint.
\end{lemma}

\begin{proof}  Take $g\in \theta_0(\overline U) \subset GL(\mathfrak n)$ and consider the value of the moment map at $g\cdot \theta_0$.  Applying the formula for the moment map at the end of Section \ref{section: the moment map}, we have the following
	\begin{eqnarray*}
	m(g\cdot \theta_0) &=& \sum_i [ (g\cdot \theta_0)(Y_i), (g\cdot \theta_0)(Y_i)^t]\\
								&=& \sum_i [ (g \theta_0 (Y_i) g^{-1}, (g^{-1})^t  \theta_0 (Y_i)^t g^t ]
	\end{eqnarray*}
where  $\{ Y_i\}$ is an orthonormal basis  of $\overline{\mathfrak{u}}$.  As $\theta_0(\overline{\mathfrak u} )$ is self-adjoint and $\overline U$ is connected, we see that $\theta_0(\overline{U})$ is self-adjoint and so $m(g\cdot \theta_0) \in \theta_0(\overline{\mathfrak{u}})$.

As the gradient of $||m||^2$ at $p\in V$ is given by $m(p)\cdot p$, using standard ODE arguments, we see that the orbit $\theta_0(\overline U) \cdot \theta_0$ is a submanifold in $V = Hom(\overline{\mathfrak{u}}, \Der(\mathfrak n) )$ which is  stable under the negative gradient flow of $||m||^2$ and the result follows.

\end{proof}

We now finish the proof of Proposition \ref{prop:  U bar acts self-adjointly on nilrad}.  The representation $\theta$ is a minimal point of the $GL(\mathfrak n)$ action on $V=Hom(\overline{\mathfrak u}, \Der(\mathfrak n))$ and from Lemma \ref{lemma: mostow on selfadjointness} we know there exists $g\in GL(\mathfrak n)$ such that $\theta_0 = g\cdot \theta$ is self-adjoint.  Recall, the orbit $GL(\mathfrak n)\cdot \theta$ is closed (Theorem \ref{thm: Rich-Slow}) and   so the negative gradient flow of $||m||^2$ (starting in the orbit) will converge to a point in the orbit which is a minimal point, see either \cite[Theorem 5.2]{JabloDistinguishedOrbits} or \cite[Section 7]{Heinzner-Schwarz-Stotzel-StratificationsforRealReductiveGroups}. (Note, these works are stated over projective space, but one easily passes back from the result in projective space to the vector space $V$.)

Thus, the trajectory $\theta_t$ of the negative gradient flow of $||m||^2$ starting at $\theta_0$ limits to a minimal point $\lambda \in GL(\mathfrak n)\cdot \theta$.  Although the representation $\theta_t$ is changing, by the above lemma the image of  $\overline{\mathfrak u} $ is not and we see that
	$$\lambda(\overline{\mathfrak{u}} ) = \theta_t(\overline{\mathfrak{u}} )= \theta_0(\overline{\mathfrak{u}}  ) $$
is self-adjoint.

Now, as $\lambda$ and our original $\theta$ are both minimal points in the same $GL(\mathfrak n)$-orbit, we know  there exists $k \in O(\mathfrak n)$ such that $\theta = k\cdot \lambda$ (Theorem \ref{thm: Rich-Slow}).  Thus, $\theta(\overline{\mathfrak{u}}) = k \lambda(\overline{\mathfrak{u}}) k^{-1}$ is self-adjoint.  This completes the proof of Proposition \ref{prop:  U bar acts self-adjointly on nilrad}.

\subsection{Constructing a  special transitive group of isometries}
Here we complete the proof of Theorem \ref{thm: reduction to nicer G1=Gnc}.

\begin{lemma}\label{lemma: K splits}  Let $G/K$ be a homogeneous Einstein space of negative scalar curvature.  Write $G=(G_1Z(U))\ltimes N$ and   $\mathfrak g = (\mathfrak g_1 + \mathfrak{z(u)} ) \ltimes \mathfrak n$, as above.  Then
	\begin{enumerate}
	\item $K=K_1K_2$, where $K_1 = K\cap G_1$ and $K_2=K\cap Z(U)$.
	\end{enumerate}
Furthermore, for $X\in \mathfrak{z(u)}$, we have
	\begin{enumerate}\setcounter{enumi}{1}
	\item $\theta(X)$ is non-zero and
	\item if $\theta(X)$ is skew-symmetric, then $X\in \mathfrak k = Lie~K$.
	\end{enumerate}
\end{lemma}

\begin{proof}  We prove (ii) and (iii) first.  If $\theta(X)$ were zero, then $X$ would be central in $\mathfrak g$ and hence contained in the nilradical $\mathfrak n$.  But $\mathfrak n \cap \mathfrak u = \emptyset$ and so $\theta(Z)$ is non-zero.  The proof of (iii) is the same as in the case $X\in \mathfrak g_c$.  See the proof of Theorem \ref{thm: main result Gc < K} which is just after Lemma \ref{lemma: theta gc is skew-symmetric}.

We prove (i).  As $K$ is connected, it suffices to prove the statement at the Lie algebra level, i.e. that $\mathfrak k = \mathfrak k_1 + \mathfrak k_2$, where $\mathfrak k_1 = \mathfrak k \cap \mathfrak g_1$ and $\mathfrak k_2 = \mathfrak k \cap \mathfrak{z(u)}$.  Take $X\in \mathfrak k$ and write it as 
	$$X= X_1 + X_2, \quad \mbox{ where }X_1\in\mathfrak g_1 \mbox{ and } X_2\in \mathfrak{z(u)}.$$
Recall from the construction of $\mathfrak{\overline{u}}$ above, we have $X_2 \in \mathfrak{z(u)} \subset \mathfrak{z(\overline{u})}$.  As $\theta(\mathfrak g_1)$ and   $\theta(\mathfrak{z(\overline{u}}))$ are self-adjoint, disjoint, and $\theta(X)$ is skew-symmetric, we see that both $\theta(X_1)$ and $\theta(X_2)$ are skew-symmetric.  Thus, $X_2\in \mathfrak k$ by (iii) and so $X_1 = X-X_2 \in \mathfrak k$.
\end{proof}

One consequence of (ii) above is that $\theta$ restricted to $\mathfrak{z(u)}$ is non-singular.  Now consider the group $\overline G$ constructed above, see Remark \ref{remark: g-bar the extended group}.  The center $\mathfrak{z(\overline{u})}$ is self-adjoint under $\theta$ and can be written as 
	$$\theta(\mathfrak{z(\overline{u})} )= (\theta(\mathfrak{z(\overline{u})}  )  \cap  \mathfrak{so(n)} ) + (\theta(\mathfrak{z(\overline{u})} )  \cap symm(\mathfrak n) ),$$
where $symm(\mathfrak n)$ denotes the symmetric matrices relative to the inner product on $\mathfrak n$.  As $\theta$ restricted to $\mathfrak{z(\overline{u})} $ is non-singular, we have a decomposition
	$$
	\mathfrak{z(\overline{u})} = \mathfrak{k_z} + \mathfrak a,
	$$
where $\theta(\mathfrak{k_z})$ consists of skew-symmetric matrices and $\theta(\mathfrak a)$ consists of symmetric matrices.  We are now in a position to prove Theorem \ref{thm: reduction to nicer G1=Gnc}.  The group $H$ below fulfills all the desired conditions.

\begin{prop} Let $H$ be the connected subgroup of $\overline G$ whose Lie algebra is given by $\mathfrak h = (\mathfrak g_{nc} + \mathfrak a)\ltimes \mathfrak n$.  Then 
	\begin{enumerate}
	\item $H$ acts transitively on $G/K = \overline G/\overline K$,
	\item the Levi factor for $H$ is $H_1=G_{nc}$ and so $H_c$ is trivial,
	\item the radical of $H$ is $H_2=AN$,
	\item the adjoint action of $\mathfrak a$ on $\mathfrak n$ is by symmetric endomorphisms, and
	\item the stabilizer $H\cap K$ of the $H$-action on $G/K$ is   contained in $H_1$.
	\end{enumerate}
\end{prop}

\begin{proof}  We prove (i).  Observe that we have the following equalities
	$$\mathfrak h + \overline{\mathfrak k} 
	=  \mathfrak g_{nc} + \mathfrak a + \mathfrak n +\overline{\mathfrak k} 
	=  \mathfrak g_{nc} + \mathfrak g_c + \mathfrak{k_z}+ \mathfrak a + \mathfrak n +\overline{\mathfrak k} 
	=  \overline{\mathfrak g} + \overline{\mathfrak k}, $$
since $\mathfrak g_c$ and $\mathfrak{k_z}$ are contained in $\overline{\mathfrak k}$.  As such, the $H$-orbit of $eK$ in $G/K=\overline G/\overline K$ is open.  As this orbit is a Riemannian homogeneous space, it is an open complete submanifold of the complete, connected Riemannian manifold $G/K$ and hence equals $G/K$.  Thus, $H$ acts transitively.

Statement (ii) is immediate from the construction of $\mathfrak h$ as the Levi factor is the maximal semi-simple subalgebra/subgroup.  Again, statement (iii) is immediate from the construction of $\mathfrak h$ as the radical is the maximal solvable ideal of $\mathfrak h$.  
By the construction of $\mathfrak a$, $\theta(\mathfrak a)$ consists of symmetric endomorphisms of $\mathfrak n$; this is precisely statement (iv).  Finally, to see (v), we apply Lemma \ref{lemma: K splits} by observing that no element of $\theta(\mathfrak{a})$ is skew-symmetric and hence $A\cap K$ is trivial.

\end{proof}

\begin{remark}  It seems interesting to note that arguments above can be used to show that for any transitive group of isometries $G$ acting on a homogeneous, Einstein space $G/K$ of negative scalar curvature that $G=G_1\ltimes G_2$ and $G_2$ decomposes as $K_2S_2$ where $K_2\subset K$, $S_2\cap K$ is trivial, and $G_1S_2$ acts transitively on $G/K$.  Furthermore, the induced geometry on $S_2$ can be shown to be Einstein.
\end{remark}

\section{The Generalized Alekseevskii Conjecture in dimension 5}\label{sec: GAC in low dim}
We apply the work above to show that any $5$-dimensional, homogeneous Ricci soliton with negative cosmological constant is isometric to a simply-connected solvable Lie group with left-invariant metric.  This verifies the Generalized Alekseevskii Conjecture in dimension $5$.  This result was recently obtained in \cite{Arroyo-Lafuente:HomogeneousRicciSolitonsInLowDimensions}  and we give an alternative proof.

We begin by restricting our attention to those Ricci solitons which are not Einstein as the Einstein case was previously established in \cite{Nikonorov:NoncompactHomogEinstein5manifolds}.  Further, we are able to restrict ourselves to the case that $G/K$ is simply-connected as we will show that the spaces of interest are solvmanifolds.  See \cite{Jablo:StronglySolvable} for this reduction to the simply-connected case.  

From \cite{Jablo:HomogeneousRicciSolitonsAreAlgebraic} together with either \cite{HePetersenWylie:WarpProdEinsteinMetricsOnHomogAndHomogRicciSolitons} or \cite{LauretLafuente:StructureOfHomogeneousRicciSolitonsAndTheAlekseevskiiConjecture}, we know that there exists $G' > G$ such that $G$ is codimension 1 in $G'$ and  $G'/K$ is Einstein with non-trivial mean curvature vector $H$ which satisfies
	$$  \ip{H,X} = tr~(ad~X) \quad \mbox{ for all } X\in\mathfrak g.$$

\begin{lemma} If $G'/K$ is a solvmanifold, then   $G/K$ is also a solvmanifold.
\end{lemma}

To see this, observe that $G'/K$ being a solvmanifold and Einstein implies the isometry group of $G'/K$ is linear \cite{Jablo:StronglySolvable}, so $G_1$ is linear and $K_1 = K\cap G_1$ is a maximal compact subgroup of $G_1$.    Let $S_1$ be the Iwasawa subgroup of $G_1$.  Recall, $S_1$ is solvable and $G_1=S_1K_1$.  From here we see that $S_1 \ltimes G_2$ acts on $G/K$ with an open orbit and so acts transitively, i.e. $G/K$ is a solvmanifold.

\begin{remark}  From the lemma above, to obtain the desired result on 5-dimensional Ricci solitons, it suffices to prove the result for  6-dimensional, homogeneous Einstein spaces of negative scalar curvature with non-trivial mean curvature vector.  This is the result we will show; in the sequel $G/K$ will denote such a 6-dimensional space.
\end{remark}

From Theorem \ref{thm: reduction to nicer G1=Gnc}  we have the decomposition 
	$$\mathfrak g = (\mathfrak g_1 + \mathfrak{z(u)} ) \ltimes \mathfrak n$$
where $\mathfrak u = \mathfrak g_1 + \mathfrak{z(u)}$ is reductive,  $\mathfrak g_1$ is semi-simple with no compact ideals, and $\mathfrak k \subset \mathfrak g_1$.  Further, we have that the mean curvature vector is central in $\mathfrak u$, i.e. $H \in \mathfrak{z(u)}$ (see Eqn.~2.1 of \cite{Jablo:HomogeneousRicciSolitonsAreAlgebraic}).

\begin{remark} To show that $G/K$ is a solvmanifold, it suffices to show that $G_1/K$ is a solvmanifold.
\end{remark}

The nilradical $\mathfrak n$ cannot be trivial as otherwise we would have an Einstein metric on $U/K$ when $U$ has non-trivial center (cf.~Corollary \ref{cor: negative ricci implies no center}).  Now we see that $\dim  G_1/K  \leq 4$.

\medskip

\textbf{Case: $\dim \mathfrak n=1$.}  In this case, we have $Der(\mathfrak n) \simeq \mathbb R$  is spanned by the mean curvature vector $H$ and so  $\theta(\mathfrak g_1) = 0$.  If $\dim G_1/K =4$, then from Lafuente-Lauret's structure theorem  (Theorem \ref{thm: LL Einstein metrics structure})  we see that $G_1/K$ is Einstein and so our result is true as all 4-dimensional non-compact, homogeneous Einstein spaces are solvmanifolds.  (See \cite{Jensen:HomogEinsteinSpacesofDim4} for the classification of 4-dimensional, non-compact, homogeneous Einstein spaces.)

Now assume that $\dim G_1/K = 3$.  In this case, as $\dim \mathfrak{z(u)}=2$, we  consider the codimension 1 subgroup $\overline G$ of $G$ with Lie algebra 
	$$ \overline{\mathfrak g} = \mathfrak g_1 + \mathbb R (X), \mbox{ where } X\in \mathfrak{z(u)} \mbox{ and } X\perp H.$$
As $tr~(ad~X) = 0$ (from the definition of the mean curvature vector $H$), we see that $\theta(X)=0$.  Further, we have that 
	$$\Ric_{\overline{G}/K} < 0$$
To see this, one applies \cite[Lemma 4.2]{LauretLafuente:StructureOfHomogeneousRicciSolitonsAndTheAlekseevskiiConjecture} together with Lafuente-Lauret's structure theorem (Theorem \ref{thm: LL Einstein metrics structure}) and the observation that $G/K = \overline{G}/K \times \mathbb R$ is a Riemannian product, where $\mathbb R$ is the Lie group whose algebra contains the mean curvature vector $H$.  However, the group $\overline G$ has center and this violates Corollary \ref{cor: negative ricci implies no center}.

\textbf{Case: $\dim \mathfrak n=2$.}  Here $\dim G_1/K =3$ and the only case in which $G_1/K$ is not already a solvmanifold is when   $\mathfrak g_1 = \mathfrak{sl}_2\mathbb R$ with $\mathfrak k$ trivial.  We demonstrate that this case is not a possibility. 

By Proposition \ref{prop:  U bar acts self-adjointly on nilrad}, we know that  $\theta(\mathfrak{sl}(2,\mathbb R))$ must be self-adjoint.  As such, there exists $X\in\mathfrak{sl}(2,\mathbb R)$ such that $\theta(X)$ is skew-symmetric and so  the smallest eigenvalue of 
	$$\Ric_{\widetilde{SL(2,\mathbb R)}} (v,v)= Ric_{U/K} (v,v) = c|v|^2 + tr( S(\theta(v)) )^2  $$
occurs in the direction of $X$ which is tangent to a maximal compact subalgebra of $\mathfrak{sl}(2,\mathbb R)$.  Note that $ad~X$ has purely imaginary eigenvalues.

One can see that  no such metric exists on $\widetilde{SL(2,\mathbb R)} $ by applying \cite{Milnor:LeftInvMetricsonLieGroups}.  From that work, we know that for any left-invariant  metric on $\widetilde{SL(2,\mathbb R)}$ there exists an orthonormal basis $\{e_1,e_2,e_3\}$ such that
	$$[e_2,e_3] = \lambda_1 e_1  \quad  [e_3,e_1] = \lambda_2 e_2 \quad [e_1,e_2] = \lambda_3 e_3$$
with $\lambda_1 < 0 < \lambda_2 \leq \lambda _3$.  Furthermore, the given basis diagonalizes $\Ric_{\widetilde{SL(2,\mathbb R)}}$.

For any left-invariant metric on $\widetilde{SL(2,\mathbb R)}$,  we claim that the smallest eigenvalue of $\Ric_{\widetilde{SL(2,\mathbb R)}}$ occurs in either the $e_2$ or $e_3$ direction.  This follows quickly from \cite{Milnor:LeftInvMetricsonLieGroups} by using that either two of the eigenvalues of $\Ric_{\widetilde{SL(2,\mathbb R)}}$ are zero and one negative, or two are negative and one positive.  In the first case, one can show that the eigenvalues are $ric(e_1)=ric(e_3)=0$ and $ric(e_2)=2\lambda_3\lambda_1 <0$.  In the case that all eigenvalues of $\Ric_{\widetilde{SL(2,\mathbb R)}}$ are non-zero, these eigenvalues are given by $ric(e_i) = 2\mu_{i+1}\mu_{i+2}$ with $\mu_i = \frac{1}{2} (-\lambda_i + \lambda_{i+1} + \lambda_{i+2})$.  (Here the formulas are written using the convention that our indices are taken mod 3.)  By inspection, one is able to see that $ric(e_1)$ cannot be the smallest such eigenvalue.

Finally, $ad~e_2$ and $ad~e_3$ have real eigenvalues.  Together with the above work we see that $ad~X$ above must have only zero eigenvalues and so $X$ is central in $\mathfrak{sl}(2,\mathbb R)$.  This is a contradiction as semi-simple Lie algebras have no center.

\section{Acknowledgments}  
The first author would like to thank J.~Lauret for stimulating conversations on the topic of this paper.  The first author was supported in part by NSF grant DMS-1105647.   The second author was supported in part by NSF grant DMS-1006677.

\providecommand{\bysame}{\leavevmode\hbox to3em{\hrulefill}\thinspace}
\providecommand{\MR}{\relax\ifhmode\unskip\space\fi MR }
\providecommand{\MRhref}[2]{%
  \href{http://www.ams.org/mathscinet-getitem?mr=#1}{#2}
}
\providecommand{\href}[2]{#2}

\end{document}